\documentclass[12pt]{amsart}
\usepackage{a4wide,enumerate,xcolor}
\usepackage{amsmath,graphicx}
\allowdisplaybreaks

\usepackage[colorlinks=true,linkcolor=blue,
anchorcolor=black,citecolor=cyan,filecolor=black,menucolor=black,
runcolor=black,urlcolor=blue]{hyperref}

\let\pa\partial
\let\na\nabla
\let\eps\varepsilon
\newcommand{\N}{{\mathbb N}}
\newcommand{\R}{{\mathbb R}}
\newcommand{\diver}{\operatorname{div}}

\newcommand{\dom}{\mathcal{D}}

\newtheorem{theorem}{Theorem}
\newtheorem{lemma}[theorem]{Lemma}

\begin{document}

\title[Cross-diffusion system for vesicle transport]{
Existence analysis of a cross-diffusion system \\
with nonlinear Robin boundary conditions \\
for vesicle transport in neurites}

\author[M. Fellner]{Markus Fellner}
\address{Institute of Analysis and Scientific Computing, Technische Universit\"at Wien,
Wiedner Hauptstra\ss e 8--10, 1040 Wien, Austria}
\email{markus.fellner@tuwien.ac.at}

\author[A. J\"ungel]{Ansgar J\"ungel}
\address{Institute of Analysis and Scientific Computing, Technische Universit\"at Wien,
Wiedner Hauptstra\ss e 8--10, 1040 Wien, Austria}
\email{juengel@tuwien.ac.at}

\date{\today}

\thanks{The authors acknowledge partial support from
the Austrian Science Fund (FWF), grants P33010 and F65.
This work has received funding from the European
Research Council (ERC) under the European Union's Horizon 2020 research and innovation programme, ERC Advanced Grant no.~101018153.}

\begin{abstract}
A one-dimensional cross-diffusion system modeling the transport of vesicles in neurites is analyzed. The equations are coupled via nonlinear Robin boundary conditions to ordinary differential equations for the number of vesicles in the reservoirs in the cell body and the growth cone at the end of the neurite. The existence of bounded weak solutions is proved by using the boundedness-by-entropy method. Numerical simulations show the dynamical behavior of the concentrations of anterograde and retrograde vesicles in the neurite. 
\end{abstract}

\keywords{Neurite growth, vesicle transport, cross-diffusion equations, nonlinear Robin boundary conditions, existence analysis.}

\subjclass[2000]{35A05, 35K51, 35K65, 35Q92, 92C20.}

\maketitle


\section{Introduction}

The aim of this paper is the analysis of cross-diffusion systems modeling the intracellular transport of vesicles in neurites. Compared to previous works like \cite{Jue15}, where no-flux boundary conditions are imposed, the novelties are the nonlinear Robin boundary conditions and the coupling to ordinary differential equations. 

Neurite growth is a fundamental process to generate axons and dendritic trees that connect to other neurons. During their development, neurites show periods of extension and rectraction until neuron polarity is established. Then one of the neurites becomes the axon, while the other neurites do not grow further. The process of elongation and retraction depends, besides many other mechanisms \cite{OlGo22}, on the motor-driven transport of vesicles inside the neurites. Vesicles are biological structures consisting of liquid or cytoplasm and are enclosed by a lipid membrane. They are produced in the cell body (soma) and transport material to the tip of a neurite (the so-called growth cone). Vesicles that fuse with the plasma membrane of the growth cone deliver their membrane lipids to the tip, causing the neurite shaft to grow. Vesicles moving to the growth cone are called anterograde vesicles. Retrograde vesicles are generated via endocytosis at the growth cone plasma membrane and move back in the direction of the soma.

We model anterograde and retrograde vesicles as two different particle species as in \cite{HDPP21}. Because of the finite size of the vesicles, we take into account size exclusion effects. In the diffusion limit of a deterministic lattice model, the authors of \cite{HDPP21} derived formally mass balance equations with fluxes that depend on the gradients of both the concentrations of the anterograde and retrograde vesicles, leading to cross-diffusion equations. The dynamics of the vesicle concentrations in the neurite pools at the soma and growth cone are governed by ordinary differential equations, which are linked to the cross-diffusion equations through nonlinear Robin boundary conditions. 

The dynamics of the concentrations (or volume fractions) of the anterograde vesicles $u_1(x,t)$ and the retrograde vesicles $u_2(x,t)$ along the one-dimensional neurite is governed by 
\begin{align}
  & \pa_t u_1 + \pa_x J_1 = 0, \quad
  J_1 = -D_1\big(u_0\pa_x u_1 - u_1\pa_x u_0 - u_0u_1\pa_x V_1\big), 
  \label{1.eq1} \\
  & \pa_t u_2 + \pa_x J_2 = 0, \quad
  J_2 = -D_2\big(u_0\pa_x u_2 - u_2\pa_x u_0 - u_0u_2\pa_x V_2\big),
  \label{1.eq2}
\end{align}
solved in the bounded interval $\Omega=(0,1)$ with the soma at $x=0$ and the growth cone at $x=1$ for times $t>0$, supplemented with the initial conditions
\begin{equation}\label{1.ic}
  u_1(\cdot,0) = u_1^0, \quad u_2(\cdot,0) = u_2^0\quad
  \mbox{in }\Omega.
\end{equation}
Here, $u_0=1-u_1-u_2$ describes the void volume fraction, $J_i$ are the corresponding fluxes, $D_i$ the diffusion coefficients, and $V_i$ given potentials. Equations \eqref{1.eq1}--\eqref{1.eq2} form a cross-diffusion system with a nonsymmetric and generally not positive definite diffusion matrix, given by
\begin{equation}\label{1.A}
  A(u) = \begin{pmatrix}
  D_1(1-u_2) & D_1 u_1 \\ D_2 u_2 & D_2(1-u_1)
  \end{pmatrix}.
\end{equation}
Moreover, if $u_0=0$, the equations are of degenerate type; see \eqref{2.dei2}. 

Let $\Lambda_n(t)/\Lambda_n^{\rm max}$ and $\Lambda_s(t)/\Lambda_s^{\rm max}$ be the percentage of currently occupied space in the soma and the growth cone, respectively. Anterograde vesicles leave the soma and enter the neurite at $x=0$ if there is enough space with rate $\alpha_1 (\Lambda_s/\Lambda_s^{\rm max})u_0(0,\cdot)$, and they enter the growth cone with rate $\beta_1(1-\Lambda_n/\Lambda_n^{\rm max}) u_0(1,\cdot)u_1(\cdot,1)$. Retrograde vesicles enter the soma with rate $\beta_1(1-\Lambda_s/\Lambda_s^{\rm max})u_0(1,\cdot)u_2(\cdot,0)$ and 
leave the growth cone with rate $\alpha_2(\Lambda_n/\Lambda_n^{\rm max})u_0(1,\cdot)$, where $\alpha_i,\beta_i>0$ for $i=1,2$ are some constants. Thus, the fluxes at $x=0$ and $x=1$ are given by the nonlinear Robin boundary conditions
\begin{align}
  J_1(0,t) &= J_1^0[u](t) := \alpha_1\frac{\Lambda_s(t)}{\Lambda_s^{\rm max}} u_0(0,t), \label{1.bc10} \\
  J_1(1,t) &= J_1^1[u](t) := \beta_1\bigg(1-\frac{\Lambda_n(t)}{\Lambda_n^{\rm max}}\bigg)u_0(1,t)u_1(1,t), \label{1.bc11} \\
  J_2(0,t) &= J_2^0[u](t) := -\beta_2\bigg(1-\frac{\Lambda_s(t)}{\Lambda_s^{\rm max}}\bigg)u_0(0,t)u_2(0,t), 
  \label{1.bc20} \\
  J_2(1,t) &= J_2^1[u](t) := -\alpha_2\frac{\Lambda_n(t)}{\Lambda_n^{\rm max}}u_0(1,t) 
  \quad\mbox{for }t>0,\label{1.bc21}
\end{align}
where $u=(u_1,u_2)$. Compared to \cite{HDPP21}, the boundary conditions \eqref{1.bc11} and \eqref{1.bc20} depend on $u_0$ to account for the resistance of entering the growth cone and soma, respectively, for instance due to viscosity. There is also a mathematical reason for this choice, which is explained below. 

Finally, the change of vesicle numbers in the soma and growth cone is determined by the corresponding in- and outflow fluxes,
\begin{align}\label{1.ode1}
  \pa_t\Lambda_n = J_1^1[u]+J_2^1[u], \quad t>0, \quad
  \Lambda_n(0) = \Lambda_n^0,\\
  \pa_t\Lambda_s = -(J_1^0[u]+J_2^0[u])\quad t>0, \quad
  \Lambda_s(0) = \Lambda_2^0. \label{1.ode2}
\end{align}
Inserting \eqref{1.bc10}--\eqref{1.bc21} into these equations, they become linear ordinary differential equations in $\Lambda_n$ and $\Lambda_n$, coupled to equations \eqref{1.eq1}--\eqref{1.eq2}.

Model \eqref{1.eq1}--\eqref{1.bc21} can be derived in the diffusion limit from the lattice model of \cite{HDPP21}; see Section \ref{sec.deriv}. A Fokker--Planck equation for single-species vesicles with in- and outflow boundary conditions was analyzed in \cite{BHP20}. The work \cite{BrKa16} models a limited transport capacity inside the neurites by taking into account size exclusion effects for a single motor-cargo complex with and without vesicles. Advection-diffusion equations for the bidirectional vesicular transport were derived in \cite{BrLe15}. Dynamically varying neurite lengths are allowed in \cite{MPW23}, leading to drift-diffusion-reaction equations. A lattice model for the probability that a receptor traveling with a vesicle is located at a given cell was analyzed in \cite{Bre06}. This model was generalized in \cite{Bre16} by allowing motor-complexes to carry an arbitrary number of vesicles, which leads to Becker--D\"oring equations for aggregation-fragmentation processes. The size of the cargo vesicles, which strongly influences the speed of retrograde transport, was taken into account in \cite{RaOe23}, and a free-boundary problem for the radius of the vesicle has been formulated. We also mention the paper \cite{BaEh18} for a related cross-diffusion system with free boundary and nonvanishing flux boundary conditions. 

The goal of this paper is to analyze model \eqref{1.eq1}--\eqref{1.bc21} mathematically. Equations \eqref{1.eq1}--\eqref{1.eq2} are similar to the ion-transport model in \cite{GeJu18}. The analysis of this system was based on the boundedness-by-entropy method \cite{BDPS10,Jue15} and a version of the Aubin--Lions compactness lemma which takes into account the degeneracy at $u_0=0$ \cite{ZaJu17}. The main difficulty here is the treatment of the nonlinear Robin boundary conditions. Linear Robin boundary conditions were considered in \cite{BuPi16} but for stationary drift-diffusion equation for one species only.

The key idea of our analysis is to work with the entropy (or, more precisely, free energy)
\begin{align}
  & E(u) = \int_\Omega(h(u)-u_1V_1-u_2V_2)dx, \quad\mbox{where} 
  \nonumber \\
  & h(u) = \sum_{i=1}^2 u_i(\log u_i-1) + u_0(\log u_0-1)
  \quad\mbox{and}\quad u_0=1-u_1-u_2. \label{1.h}
\end{align}
Introducing the electrochemical potentials
$\mu_i=\delta E/\delta u_i=\log(u_i/u_0)-V_i$ for $i=1,2$, system \eqref{1.eq1}-\eqref{1.eq2} can be written as a formal gradient flow in the sense
$$
  \pa_t u_i = \diver\sum_{j=1}^2 B_{ij}\na\mu_j, \quad\mbox{where }
  B_{ij}=D_iu_0u_i\delta_{ij}, \ i=1,2,
$$
and $\delta_{ij}$ is the Kronecker symbol. The advantage of this formulation is that the drift terms are eliminated and that the new diffusion matrix $(B_{ij})$ is (diagonal and) positive definite. This formulation is the basis of the boundedness-by-entropy method \cite[Chap.~4]{Jue16}.
The use of the electrochemical potentials has another benefit. Inverting the relation $(u_1,u_2)\mapsto(\mu_1,\mu_2)$, we infer from
$$
  u_i = \frac{\exp(\mu_i+V_i)}{1+\exp(\mu_1+V_1)+\exp(\mu_2+V_2)}, \quad i=1,2,
$$
that
\begin{equation}\label{1.dom}
  u=(u_1,u_2)\in \dom := \big\{u\in\R^2:u_1>0,\,u_2>0,\, u_1+u_2<1\big\},
\end{equation}
guaranteeing the physical bounds without the use of a maximum principle. 

Furthermore, a formal computation (see the proof of \eqref{2.dei}) 
shows that
\begin{align}\label{1.ei}
  \frac{dE}{dt}(u) + \int_\Omega\sum_{i=1}^2 D_i u_0u_i\bigg|\na
  \bigg(\log\frac{u_i}{u_0}-V_i\bigg)\bigg|^2 dx
  = -\sum_{i=1}^2\bigg[J_i\cdot\nu\bigg(\log\frac{u_i}{u_0}-V_i\bigg)
  \bigg]_{x=0}^{x=1},
\end{align}
where $\nu(0)=-1$ and $\nu(1)=1$. The most delicate terms are
$J_1\cdot\nu(\log(u_1/u_0)-V_1)|_{x=0}$ and $J_2\cdot\nu(\log(u_2/u_0)-V_2)|_{x=1}$. To estimate these expressions, we exploit the fact that both terms factorize $u_0$. For instance,
$$
  -J_{1}\cdot\nu\bigg(\log\frac{u_i}{u_0}-V_i\bigg)\bigg|_{x=0}
  = \alpha_1\frac{\Lambda_s}{\Lambda_s^{\max}}u_0
  (\log u_1 - \log u_0 - V_1)|_{x=0} 
$$
is bounded from above since $\Lambda_{s}\ge 0$, $-u_0\log u_0$ is bounded, and $u_0\log u_1$ is nonpositive due to $0<u_1<1$. Similarly, the other boundary terms are bounded, and we conclude that the right-hand side of \eqref{1.ei} is bounded from above. An estimation of the entropy production term (the second term on the left-hand side of \eqref{1.ei}) shows that
(see, e.g., the proof of Lemma 6 in \cite{GeJu18})
\begin{align}\label{1.ep}
  \int_\Omega&\sum_{i=1}^2 D_i u_0u_i\bigg|\na
    \bigg(\log\frac{u_i}{u_0}-V_i\bigg)\bigg|^2 dx \\
    &\ge c\int_\Omega\bigg(\sum_{i=1}^2 u_0|\na\sqrt{u_i}|^2 
    + |\na\sqrt{u_0}|^2\bigg)dx - C\int_\Omega\sum_{i=1}^2|\na V_i|^2 dx.
    \nonumber
\end{align}
Together with the $L^\infty(\Omega)$ bounds for $u_i$, this provides $H^1(\Omega)$ bounds for $u_0$ and $u_0u_i$ for $i=1,2$, which are needed to apply the ``degenerate'' Aubin--Lions lemma \cite{Jue15}. Moreover, the bounds show that we can define the traces of $u_0u_i$ and $u_0$, which is needed to give a meaning to the boundary conditions \eqref{1.bc10}--\eqref{1.bc21}. At this point, we need the factor $u_0u_i$ in \eqref{1.bc11} and \eqref{1.bc20}. Indeed, without the factor $u_0$, we are not able to define $u_1$ and $u_2$ at $x=0,1$.
This is the mathematical reason to introduce this factor. 

We note that our method also works for more than two species and in several space dimensions. Thanks to the $L^\infty(\Omega)$ bounds, no restriction on the space dimension due to Sobolev embeddings is needed. For more than two species, one may apply the techniques elaborated in \cite{GeJu18}.

For our main result, we impose the following assumptions:
\begin{itemize}
\item[(A1)] Domain: $\Omega=(0,1)$, $T>0$, $\Omega_T:=\Omega\times(0,T)$.
\item[(A2)] Parameter: $\alpha_i$, $\beta_i$, $D_i>0$, $V_i\in H^1(\Omega)$ for $i=1,2$ and $\Lambda_n^{\rm max}$, $\Lambda_s^{\rm \max}>0$.
\item[(A2)] Initial data: $u_1^0,u_2^0\in L^1(\Omega)$ satisfies $(u_1^0,u_2^0)(x)\in\dom$ for a.e.\ $x\in\Omega$ (see Definition \ref{1.dom} of $\dom$) and $\Lambda_n^0/\Lambda_n^{\max}$, $\Lambda_s^0/\Lambda_s^{\max}\in[0,1]$.
\end{itemize}

\begin{theorem}[Global existence]\label{thm.ex}
Let Assumptions (A1)--(A3) hold. Then there exists a weak solution
$(u_1,u_2,\Lambda_n,\Lambda_s)$ to \eqref{1.eq1}--\eqref{1.ode2} satisfying $u_1,u_2\ge 0$ and $u_1+u_2\le 1$ in $\Omega_T$,
$$
  \sqrt{u_0}u_i,\,\sqrt{u_0}\in L^2(0,T;H^1(\Omega)), \quad
  \pa_t u_i\in L^2(0,T;H^1(\Omega)'), \quad i=1,2,
$$ 
the weak formulation
$$
  \int_0^T\langle\pa_t u_i,\phi_i\rangle dt
  - \int_0^T\int_\Omega J_i\pa_x\phi_i dxdt
  + \int_0^T\big[J_i(x,t)\phi_i(x,t)\big]_{x=0}^{x=1}dt = 0
$$
where $\langle\cdot,\cdot\rangle$ is the dual product between $H^1(\Omega)'$ and $H^1(\Omega)$, the fluxes are defined as
$$
  J_i = \sqrt{u_0}\pa_x(\sqrt{u_0}u_i) - 3\sqrt{u_0}u_i\pa_x\sqrt{u_0}
  - u_0u_i\pa_x V_i\in L^2(\Omega_T), \quad i=1,2,
$$ 
the initial conditions \eqref{1.ic} are satisfied in the sense of $H^1(\Omega)'$, and equations \eqref{1.ode1}--\eqref{1.ode2} are fulfilled in the sense of $L^2(\pa\Omega)$.
\end{theorem}

As mentioned above, the regularity of $u_0$ and $u_i$ for $i=1,2$ allows us to define the trace of $u_0$ and $u_0u_i$ such that the boundary conditions and the differential equations for $\Lambda_n$ and $\Lambda_s$ are well defined. The proof of Theorem \ref{thm.ex} is based on the entropy identity \eqref{1.ei}, a regularization of equations \eqref{1.eq1}--\eqref{1.eq2}, the Leray--Schauder fixed-point theorem, and a compactness argument using uniform gradient estimates coming from \eqref{1.ep}. 

The paper is organized as follows. We sketch the formal derivation of \eqref{1.eq1}--\eqref{1.bc21} from a lattice model in Section \ref{sec.deriv}. Theorem \ref{thm.ex} is proved in Section \ref{sec.ex}. We present in Section \ref{sec.num} some numerical experiments and properties of stationary solutions. 


\section{Formal derivation of the model}\label{sec.deriv}

Equations \eqref{1.eq1}-\eqref{1.bc21} can be formally derived from discrete dynamics on a lattice, which takes into account the in- and outflow of vesicles into the respective lattice cell. The derivation is similar to the presentation in \cite{HDPP21}; we repeat it for the convenience of the reader and to highlight the main difference to \cite{HDPP21}. We divide the domain $\Omega=(0,1)$ into $m$ cells $K_j$ of length $h>0$ and midpoint $x_{j}=hj$, where $j=0,\ldots,m-1$. The cell $K_j$ is occupied by anterograde vesicles with volume fraction $u_{1,j}(t)=u_{1}(x_{j},t)$ and retrograde vesicles with volume fraction $u_{2,j}(t)=u_2(x_{j},t)$. 

The transition rate of a vesicle to jump from cell $j$ to a neighboring cell $j\pm 1$ equals 
$$
  u_{i,j}u_{0,j\pm 1}\exp[-\eta_i(V_i(x_{j\pm 1})+V_i(x_j))], \quad i=1,2,
$$
where $\eta_i>0$ is some constant and $V_{i,j}=V_i(x_j,\cdot)$, taking into account that a jump is possible only if the cell $j$ is not empty ($u_{i,j}>0$) and the cell $j\pm 1$ is not fully occupied ($u_{0,j\pm 1}>0$). The dynamics of $u_{i,j}$ is then given by
\begin{align}\label{aux}
  \gamma_i h^2\pa_t u_{i,j}
  &= -u_{i,j}u_{0,j-1}e^{-\eta_i(V_{i,j}-V_{i,j-1})}
  + u_{i,j-1}u_{0,j}e^{-\eta_i(V_{i,j-1}-V_{i,j})} \\
  &\phantom{xx}- u_{i,j}u_{0,j+1}e^{-\eta_i(V_{i,j}-V_{i,j+1})}
  + u_{i,j+1}u_{0,j}e^{-\eta_i(V_{i,j+1}-V_{i,j})}, \nonumber
\end{align}
where $\gamma_i>0$. The factor $h^2$ on the left-hand side corresponds to a diffusion scaling. By Taylor expansion, we have $e^{-\eta_i z}=1-\eta_i z + \eta_i^2 z^2/2 + O(z^3)$ and $V_{i,j}-V_{i,j-1}=h\pa_x V_{i,j-1/2}+O(h^3)$, where $V_{i,j\pm 1/2}=V_i((j\pm 1/2)h,\cdot)$. Then
\begin{align*}
  e^{-\eta_i(V_{i,j}-V_{i,j-1})} &= 1 - \eta_i h\pa_x V_{i,j-1/2}
  + \eta_i^2h^2(\pa_x V_{i,j-1/2})^2 + O(h^3), \\
  e^{-\eta_i(V_{i,j-1}-V_{i,j})} &= 1 + \eta_i h\pa_x V_{i,j-1/2}
    + \eta_i^2h^2(\pa_x V_{i,j-1/2})^2 + O(h^3).
\end{align*}
In a similar way, we expand $u_{i,j\pm 1}=u_{i,j} \pm h\pa_x u_{i,j} 
+ (h^2/2)\pa_x^2 u_{i,j} + O(h^3)$. Inserting these expansions into \eqref{aux}, we find after a computation that
\begin{align*}
  \gamma_i h^2\pa_t u_{i,j} 
  &= (u_{0,j}\pa_x^2 u_{i,j}-u_{i,j}\pa_x^2 u_{0,j})h^2
  - 2\eta_i u_{i,j}u_{0,j} (\pa_x V_{i,j+1/2}-\pa_x V_{i,j-1/2})h \\
  &\phantom{xx}- \eta_i(u_{0,j}\pa_x u_{i,j} + u_{i,j}\pa_x u_{0,j})
  (\pa_x V_{i,j+1/2}+\pa_x V_{i,j-1/2})h^2 
  + O(h^3) \\
  &= (u_{0,j}\pa_x^2 u_{i,j}-u_{i,j}\pa_x^2 u_{0,j})h^2
  -2\eta_i u_{i,j}u_{0,j}\pa_x^2 V_{i,j}h^2 \\
  &\phantom{xx}- 2\eta_i(u_{0,j}\pa_x u_{i,j} 
  + u_{i,j}\pa_x u_{0,j})\pa_x V_{i,j}h^2 + O(h^3),
\end{align*}
where we expanded $h\pa_x V_{i,j\pm 1/2}= h\pa_x V_{i,j}\pm (h^2/2)\pa_x^2 V_{i,j}+O(h^3)$. We divide this equation by $h^2$, and pass to the formal limit $h\to 0$:
\begin{align*}
  \gamma_i\pa_t u_{i} &= (u_0\pa_x^2 u_i - u_i\pa_x^2 u_0) 
  - 2\eta_i u_iu_0\pa_x^2V_i
  - 2\eta_i(u_0\pa_x u_i+u_i\pa_x u_0)\pa_x V \\
  &= \pa_x\big(u_0\pa_x u_i - u_i\pa_x u_0 - 2\eta_iu_0u_i\pa_x V).
\end{align*}
Setting $\eta_i=1/2$ and $D_i=1/\gamma_i$, we obtain \eqref{1.eq1}--\eqref{1.eq2}.

At the points $x=0$ and $x=1$, there are reservoirs with concentrations $\Lambda_s$ at $x=0$ and $\Lambda_n$ at $x=1$. The in- and outflow rates are given by
$$
  A_i(\Lambda_{\ell}) = a_i u_{0,0}
  \frac{\Lambda_{\ell}}{\Lambda_{\ell}^{\rm max}},
  \quad B_i(\Lambda_{\ell}) = b_i u_{0,m}
  \bigg(1-\frac{\Lambda_{\ell}}{\Lambda_{\ell}^{\rm max}}\bigg),
  \quad \ell=n,s,
$$
where $a_i,b_i>0$. 
We have multiplied these rates by the factor $u_{0,j}$ with $j=0$ and $j=m$, respectively, which models the resistance of entering the first and last cell. This is the main difference to the derivation in \cite{HDPP21}. Taken into account the inflow and outflow of vesicles at $x=0$, the change of the fraction of the anterograde vesicles is given by
\begin{align*}
  h^2\pa_t u_{1,0} &= -u_{1,0}(t)u_{0,1}(t)
  e^{-\eta_{1}(V_1(x_1)-V_1(x_0))} \\
  &\phantom{xx}+ u_{1,1}(t)u_{0,0}(t)e^{-\eta_{1}(V_1(x_0)-V_1(x_1))}
  + a_1\frac{\Lambda_s(t)}{\Lambda_s^{\rm max}}u_{0,0}(t)h,
\end{align*}
An expansion similarly as before, up to $O(h^2)$ instead of $O(h^3)$, leads to
\begin{align*}
  h^2\pa_t u_{1,0} 
  &= -u_{1,0}(u_{0,0} + h\pa_x u_{0,0})(1+\eta_1 h\pa_x V_{1,0}) \\
  &\phantom{xx}
  + (u_{1,0}+h\pa_x u_{1,0})u_{0,0}(1-\eta_1 h\pa_x V_{1,0}) + O(h^2) \\
  &= h(u_{0,0}\pa_x u_{1,0}
  + u_{1,0}\pa_x u_{0,0}) - 2\eta_i h u_{0,0}u_{1,0}\pa_x V_{1,0}
  + a_i h\frac{\Lambda_s}{\Lambda_s^{\rm max}} u_{0,0} + O(h^2).
\end{align*}
We divide the previous equation by $h$ and perform the limit $h\to 0$:
$$
  0 = (u_{0,0}\pa_x u_{1,0} + u_{1,0}\pa_x u_{0,0})
  - 2\eta_i u_{0,0}u_{1,0}\pa_x V_{1,0} 
  + a_i\frac{\Lambda_s}{\Lambda_s^{\rm max}} u_{0,0}.
$$
We set $\eta_i=1/2$ and $\alpha_i=a_iD_i$ and multiply the equation by $D_i$:
$$
  J_1(0,\cdot) = -D_i\big((u_{0,0}\pa_x u_{1,0} + u_{1,0}\pa_x u_{0,0}
  - u_{0,0}u_{1,0}\pa_x V_{1,0}\big) 
  = \alpha_i\frac{\Lambda_s}{\Lambda_s^{\rm max}} u_{0}(0,\cdot),
$$
which equals \eqref{1.bc10}. The boundary conditions \eqref{1.bc11}--\eqref{1.bc21} are shown in a similar way.


\section{Proof of Theorem \ref{thm.ex}}\label{sec.ex}

After proving some auxiliary lemmas, we regularize system \eqref{1.eq1}--\eqref{1.eq2} in time and space and prove the existence of a solution to this approximate problem by using the Leray--Schauder fixed-point theorem. The compactness of the fixed-point operator follows from the discrete entropy inequality analogous to \eqref{1.ei}. This inequality also provides a priori estimates uniform in the approximation parameters. The relative compactness of the sequence of approximate solutions then follows from a ``degenerate'' Aubin--Lions-type result. Finally, we verify that the limit function is a solution to \eqref{1.eq2}--\eqref{1.bc21}. To simplify the notation, we set $\Lambda_n^{\rm max}=1$ and $\Lambda_s^{\rm max}=1$ in the analysis. 

\subsection{Auxiliary lemmas}

The following lemma follows from a straightforward computation
(also see \cite[(4.61)]{Jue16}). 

\begin{lemma}\label{lem.h2A}
Let $h(u)$ be given by \eqref{1.h} and let 
$A=(A_{ij}(u))\in\R^{2\times 2}$ be defined by \eqref{1.A}. Then, for any $u\in\dom$ and $z\in\R^2$,
\begin{align*}
  z\cdot h''(u)A(u)z &= \min\{D_1,D_2\}u_0\bigg(\frac{z_1^2}{u_1}+\frac{z_2^2}{u_2}
  \bigg) + \min\{D_1,D_2\}\bigg(\frac{1}{u_0}+1\bigg)(z_1+z_2)^2 \\
  &\phantom{xx}
  + |D_2-D_1|\frac{u_2}{u_0}\bigg|z_1 - \frac{1-u_2}{u_2}z_2\bigg|^2.
\end{align*}
\end{lemma}

Let $w=h'(u)$, i.e.\ $w_i=\pa h/\pa u_i=\log(u_i/u_0)$ for $i=1,2$, and recall that $B=A(u)h''(u)^{-1}$. Then, by Lemma \ref{lem.h2A}, for some $c>0$,
$$
  \pa_x w\cdot B\pa_x w = (\pa_x u)\cdot h''(u)A(u)(\pa_x u)
  \ge c\sum_{i=1}^2 u_0(\pa_x\sqrt{u_i})^2 + c(\pa_x\sqrt{u_0})^2,
$$
which provides gradient bounds; also see \eqref{2.B} below.

\begin{lemma}\label{lem.ode}
Let $f_i,g_i\in L^2(0,T)$ be such that $f_i,g_i\ge 0$ for $i=1,2$. Then there exists a unique solution to
\begin{align}
  \pa_t\Lambda_n &= \beta_1(1-\Lambda_n)f_1(t) 
  - \alpha_2\Lambda_n g_1(t), \label{2.ode1} \\
  \pa_t\Lambda_s&= \beta_2(1-\Lambda_s)f_2(t)
  - \alpha_1\Lambda_s g_2(t), \quad t>0, \label{2.ode2}
\end{align}
with the initial conditions $\Lambda_n(0)=\Lambda_n^0\in[0,1]$ and
$\Lambda_s(0)=\Lambda_s^0\in[0,1]$ satisfying $0\le\Lambda_n(t),\Lambda_s(t)\le 1$ for $t\ge 0$. 
\end{lemma}

\begin{proof}
The existence of a unique absolutely continuous solution to the differential system \eqref{2.ode1}--\eqref{2.ode2} follows from a standard application of Banach's fixed-point theorem. We sketch the argument for the convenience of the reader.

Let $T'<T$ and
\begin{align*}
  \Gamma[\widetilde{\Lambda}](t)
  := \Lambda_{n}^{0}+\int_{0}^{t}\big[\beta_{1}(1-\widetilde{\Lambda}(s))
  f_{1}(s)-\alpha_{2}\widetilde{\Lambda}(s)g_{1}(s)\big]ds,\quad t\in[0,T'].
\end{align*} 
Exploiting the linearity with respect to $\widetilde{\Lambda}$, standard estimates show the Lipschitz continuity of $\Gamma:C^0([0,T'])\to C^0([0,T'])$:
\begin{align*}
  \|\Gamma[\widetilde{\Lambda}_{1}]-\Gamma[\widetilde{\Lambda}_{2}]
  \|_{L^\infty(0,T')}
  \leq (\alpha_{2}\|g_{1}\|_{L^1(0,T')}
  + \beta_{1}\|f_{1}\|_{L^1(0,T')})
  \|\widetilde{\Lambda}_{1}-\widetilde{\Lambda}_{2}\|_{L^\infty(0,T')}.
\end{align*}
Due to 
\begin{align*}
  \|f_{1}\|_{L^1(0,T')}+\|g_{1}\|_{L^1(0,T')}\leq \sqrt{T'}(\|f_{1}\|_{L^2(0,T)}+\|g_{1}\|_{L^2(0,T)})\rightarrow 0
\end{align*}
as $T'\rightarrow 0$, there exists some $T_{0}<T$ such that
\begin{align}\label{2.ode.cont}
  \alpha_{2}\|g_{1}\|_{L^1(0,T_{0})}
  + \beta_{1}\|f_{1}\|_{L^1(0,T_{0})}
  < \sqrt{T_{0}}(\alpha_{2}\|g_{1}\|_{L^2(0,T)}
  +\beta_{1}\|f_{1}\|_{L^2(0,T)})<1,
\end{align}
i.e., $\Gamma$ is a contraction on $C^0([0,T_{0}])$. Banach's fixed-point theorem yields a unique solution to \eqref{2.ode1} on $[0,T_{0}]$. In view of \eqref{2.ode.cont}, this procedure can be repeated on intervals $[a,b]$, satisfying $0\leq a<b\leq T$ and  $b-a<T_{0}$. Hence, the solution can be progressively extended to the whole interval $[0,T]$. Similarl,y one proceeds for (\ref{2.ode2}).   

Multiplying \eqref{2.ode1} by $\Lambda_n^-:=\max\{0,\Lambda_n\}$ yields
$$
  \frac12\frac{d}{dt}(\Lambda_n^-)^2
  = \beta_1 f_1(t)(1-\Lambda_n)\Lambda_n^- 
  - \alpha_2 g_1(t)(\Lambda_n^-)^2 \le 0,
$$
using $f_1\ge 0$ and $g_1\ge 0$. We conclude from $\Lambda_n^-(0)=0$ that $\Lambda_n(t)\ge 0$ for $t\ge 0$. In a similar way, we infer after multiplication of \eqref{2.ode2} by $(\Lambda_n-1)^+:=\max\{0,\Lambda_n-1\}$ that
$$
  \frac12\frac{d}{dt}[(\Lambda_n-1)^+]^2
  = -\beta_1 f_1(t)(\Lambda_n-1)(\Lambda_n-1)^+
  - \alpha_2 g_1(t)\Lambda_n(\Lambda_n-1)^+ \le 0,
$$
which implies that $\Lambda_n(t)\le 1$ since $\Lambda_n(0)\le 1$. 
The proof of $0\le\Lambda_s\le 1$ is similar. 
\end{proof}

\subsection{Solution of an approximate system}

The approximate system is defined by an implicit Euler discretization and a regularization in the entropy variable. Let $T>0$, $N\in\N$, $\tau=T/N$, $t_k=k\tau$ for $k=0,\ldots,N$, and $\eps>0$. 
Let $k\ge 1$ and $u^{k-1}\in L^\infty(\Omega;\R^2)$ be given. We wish to find a solution $w^k=(w_1^k,w_2^k)\in H^1(\Omega;\R^2)$ to 
\begin{align}
  \frac{1}{\tau}&\int_\Omega(u(w^k)-u^{k-1})\cdot\phi dx
  + \int_\Omega \pa_x\phi\cdot B(w^k)\pa_x w^k dx
  - \int_\Omega\sum_{i=1}^2 u_0(w^k) u_i(w^k)
  \pa_x V_i\pa_x\phi_i dx \nonumber \\
  &+ \sum_{i=1}^2\big(J_i^1[u(w^k)](t_k)\phi_i(1)-J_i^0[u(w^k)](t_k)
  \phi_i(0)\big) + \eps\int_\Omega(\pa_x w^k\cdot\pa_x\phi+ w^k\cdot\phi)dx = 0 \label{2.approx}
\end{align}
for all $\phi\in H^1(\Omega;\R^2)$. The function $u_i(w^k)$ equals $u_i(w^k)=\exp w_i^k/(1+\exp w_1^k+\exp w_2^k)$, and the entries of the matrix $B(w^k)$ are $B_{ij}(w^k)=D_i u_0(w^k)u_i(w^k)\delta_{ij}$ for $i,j=1,2$. We set $u^k:=u(w^k)$ to simplify the notation.

The pool concentrations $\Lambda_{n}^{k}$ and $\Lambda_{s}^{k}$ at iteration step $k$ are defined by $\Lambda_{j}^{k}=\Lambda_{j}(t)$ for $(k-1)\tau<t\leq k\tau$, where $\Lambda_{j}$ for $j=n,s$ are the solutions of the following fixed-point problem
\begin{align}
  \Lambda_{n}(t) &= \Lambda_{n}^{0} + \sum_{j=0}^{k-2}\bigg(\beta_{1}
  \int_{j\tau}^{(j+1)\tau}(1-\Lambda_{n}(r))u_{0}^{j}(1,r)
  u_{1}^{j}(1,r) dr
  - \alpha_{2}\int_{j\tau}^{(1+j)\tau}\Lambda_{n}(r)u_{0}^{j}(1,r) dr 
  \bigg) \nonumber \\
  &\phantom{xx}+ \beta_{1}\int_{(k-1)\tau}^{t}(1-\Lambda_{n}(r))u_{0}^{k-1}(1,r)
  u_{1}^{k-1}(1,r)dr - \alpha_{2}\int_{(k-1)\tau}^{t}\Lambda_{n}(r)
  u_{0}^{k-1}(1,r) dr, \label{2.lambda.tk} \\
  \Lambda_{s}(s) &= \Lambda_{s}^{0} + \sum_{j=0}^{k-2}\bigg(
  \beta_{2}\int_{j\tau}^{(j+1)\tau}
  (1-\Lambda_{s}(r))u_{0}^{j}(0,r)u_{2}^{j}(0,r)dr
  - \alpha_{1}\int_{j\tau}^{(1+j)\tau}\Lambda_{s}(r)u_{0}^{j}(0,r)dr
  \bigg) \nonumber \\
  &\phantom{xx}+ \beta_{2}\int_{(k-1)\tau}^{t}(1-\Lambda_{n}(r))
  u_{0}^{k-1}(0,r)u_{2}^{k-1}(0,r)dr
  - \alpha_{2}\int_{(k-1)\tau}^{t}\Lambda_{n}(r)u_{0}^{k-1}(0,r) dr.
  \nonumber
\end{align}
These equations can be interpreted as differential equations of the form
\begin{align*}
  \pa_t\Lambda_n &= \beta_1(1-\Lambda_n)f_1(t) 
  - \alpha_2\Lambda_n g_1(t), \\
  \pa_t\Lambda_s&= \beta_2(1-\Lambda_s)f_2(t)
  - \alpha_1\Lambda_s g_2(t), \quad t>0, 
\end{align*}
with suitable step functions $f_{i},g_{i}$, $i=1,2$. It follows from $|u_{0}^{k}|,|u_{i}^{k}|\leq 1$ that $f_{i},g_{i}\in L^2(0,T)$, and Lemma \ref{lem.ode} guarantees a unique solution to \eqref{2.lambda.tk}.

The variable $w_i^k=\log(u_i(w^k)/u_0(w^k))$ can be interpreted as the chemical potential, different from the electrochemical potential $\mu_i$ used in the introduction, which also includes the eletric potential $V_i$. The following analysis could also be  carried out using $\mu_i$ instead of $w_i$. 

\begin{lemma}
There exists a solution $w^k\in H^1(\Omega)$ to \eqref{2.approx} satisfying the discrete entropy inequality
\begin{align}\label{2.dei}
  H(u^k)-H(u^{k-1}) + \frac{c\tau}{2}\int_\Omega
  \pa_x w^k\cdot B(w^k)\na_x w^k dx
  + \eps\tau\int_\Omega(|\pa_x w^k|^2 + |w^k|^2)dx \le C\tau,
\end{align}
where $c=\min\{D_1,D_2\}$ and $C>0$ only depends on $\alpha_i$, $\eta_i$, $D_i$, and the $L^2(\Omega)$ norm of $|\pa_x V_i|^2$
for $i=1,2$.
\end{lemma}

\begin{proof}
The proof is similar to that one of Lemma 5 in \cite{GeJu18}, and we highlight the differences only. By the Lax--Milgram lemma, for any given $y\in H^1(\Omega;\R^2)$ and $\sigma\in[0,1]$, there exists a unique solution to the linear problem $a(v,\phi)=\sigma F(\phi)$ 
for all $\phi\in H^1(\Omega;\R^2)$, where
\begin{align*}
  a(v,\phi) &= \int_\Omega\pa_x\phi\cdot B(y)\pa_x v dx
  + \eps\int_\Omega(\pa_x w\cdot\pa_x\phi + w\cdot\phi)dx, \\
  F(\phi) &= -\frac{1}{\tau}\int_\Omega(u(y)-u^{k-1})\cdot\phi dx
  + \int_\Omega\sum_{i=1}^2 u_0(y) u_i(y)\pa_x V_i\pa_x\phi_i dx \\
  &\phantom{xx}- \sum_{i=1}^2\big(J_i^1[u(y)](t_k)\phi_i(1)
  - J_i^0[u(y)](t_k)\phi_i(0)\big)\quad\mbox{for }
  v,\phi\in H^1(\Omega;\R^2).
\end{align*}
This defines the fixed-point operator $S:C^0([0,T];\R^2)\times[0,1]
\to C^0([0,T];\R^2)$, $S(y,\sigma)=v$, where $v$ lies in fact in the space $H^1(\Omega;\R^2)$. Compared to \cite{Jue15}, we work with the space $C^0([0,T];\R^2)$ instead of $L^\infty(\Omega;\R^2)$ to ensure that the evaluation on the boundary points is well defined. By standard arguments (see, e.g., \cite[Lemma 5]{Jue15}), $S(y,0)=0$,
$S$ is continuous and compact, since the embedding $H^1(\Omega)\hookrightarrow C^0([0,T])$ is compact. It remains to prove a uniform bound for all fixed points of $S(\cdot,\sigma)$. 

We choose $\phi=v$ in $a(v,\phi)=\sigma F(\phi)$ to find that
\begin{align}\label{2.aux}
  \frac{\sigma}{\tau}&\int_\Omega(u(v)-u^{k-1})\cdot v dx
  + \int_\Omega\pa_x v\cdot B(v)\pa_x v dx 
  + \eps\int_\Omega(|\pa_x v|^2 + |v|^2)dx \nonumber \\
  &= \sigma\int_\Omega\sum_{i=1}^2 u_0(v)u_i(v)\pa_x V_i\pa_x v_i dx
  - \sigma\sum_{i=1}^2\big(J_i^1[u(v)](t_k)v_i(1)
  - J_i^0[u(v)](t_k)v_i(0)\big) \\
  &=: I_1 + I_2. \nonumber
\end{align}
The convexity of the entropy density $h$ implies that
$$
  (u(v)-u^{k-1})\cdot v
  = (u(v)-u^{k-1})\cdot h'(u(v)) 
  \ge h(u(v))-h(u^{k-1}).
$$
We conclude from Lemma \ref{lem.h2A} that
\begin{align}\label{2.B}
  \pa_x v\cdot B(v)\pa_x v
  &= \pa_x u(v)\cdot h''(u(v))A(u(v))\pa_x u(v) \\
  &\ge c\bigg(\sum_{i=1}^2u_0(v)\frac{|\pa_x u_i(v)|^2}{u_i(v)}
  + \frac{|\pa_x u_0(v)|^2}{u_0(v)}\bigg), \nonumber
\end{align}
where $c=\min\{D_1,D_2\}>0$.
For the first term on the right-hand side of \eqref{2.aux}, we observe
that the derivative of $v_i=\log(u_i/u_0)$ equals
$\pa_x v_i = \pa_x u_i(v)/u_i(v) - \pa_x u_0(v)/u_0(v)$.
Therefore, for any $\delta>0$,
\begin{align*}
  I_1 &\le \int_\Omega\sum_{i=1}^2
  \big(u_0(v)|\pa_x u_i(v)|+u_i(v)|\pa_x u_0(v)|\big)|\pa_x V_i|dx \\
  &\le \delta\int_\Omega\sum_{i=1}^2\big(u_0(v)^2|\pa_x u_i(v)|^2
  + u_i(v)^2|\pa_x u_0(v)|^2\big)dx
  + C(\delta)\int_\Omega\sum_{i=1}^2|\pa_x V_i|^2 dx \\
  &\le \delta\int_\Omega\bigg(\sum_{i=1}^2 u_0(v)
  \frac{|\pa_x u_i(v)|^2}{u_i(v)} + \frac{|\pa_x u_0(v)|^2}{u_0(v)}
  \bigg)dx + C(\delta),
\end{align*}
where we used $u_i(v)\le 1$, $u_0(v)\le 1$, and the assumption $V_i\in H^1(\Omega)$ in the last step. Choosing $\delta=c/2$, the first term on the right-hand side can be absorbed by the second term on the left-hand side of \eqref{2.aux}, thanks to \eqref{2.B}. Finally, using definitions 
\eqref{1.bc10}--\eqref{1.bc21} and $v_i=\log(u_i(v)/u_0(v))$,
\begin{align*}
  I_2 &= -\sigma\beta_1(1-\Lambda_n)u_0(v(1))u_1(v(1))
  \log\frac{u_1(v(1))}{u_0(v(1))}
  + \sigma\alpha_1\Lambda_s u_0(v(0))\log\frac{u_1(v(0))}{u_0(v(0))} \\
  &\phantom{xx}+ \sigma\alpha_2\Lambda_n 
  u_0(v(1))\log\frac{u_2(v(1))}{u_0(v(1))}
  - \sigma\beta_2(1-\Lambda_s)u_0(v(0))u_2(v(0))
  \log\frac{u_2(v(0))}{u_0(v(0))}.
\end{align*}
Since $z\mapsto z\log z$ is bounded for $z\in[0,1]$ and $\Lambda_n\le 1$, $\Lambda_s\le 1$ by Lemma \ref{lem.ode}, the first and
fourth terms on the right-hand side are bounded from above. Furthermore, we deduce from the fact that $\log u_i(v(x))$ is nonpositive for 
$i=1,2$ and $x=0,1$ that the second and third terms are nonpositive. This shows that $I_2\le C$ for some constant $C>0$ which depends only on $\alpha_i$ and $\beta_i$. 

Summarizing, \eqref{2.aux} becomes
\begin{align*}
  H(u(v)) - H(u^{k-1}) + \frac{\tau}{2}\int_\Omega 
  \pa_x v\cdot B(v)\pa_x v dx 
  + \eps\int_\Omega(|\pa_x v|^2 + |v|^2)dx \le C\tau,
\end{align*}
and $C>0$ only depends on $\alpha_i$, $\beta_i$, $D_i$, and the $L^2(\Omega)$ norm of $|\pa_x V_i|^2$ for $i=1,2$. In view of the positive semidefiniteness of $B(v)$, this inequality provides a uniform bound for $v$ in $H^1(\Omega;\R^2)$ (also being uniform in $\sigma\in[0,1]$, but not uniform in $\eps$). Hence, we can apply the fixed-point theorem of Leray and Schauder to conclude the existence of a fixed point of $S(\cdot,1)$, which is a solution to \eqref{2.approx}. Defining $w^k:=v$, this fixed point satisfies \eqref{2.dei}.
\end{proof}

Summing the discrete entropy inequality \eqref{2.dei} over $k$ leads to the following result.

\begin{lemma}
There exists $C>0$ independent of $(\eps,\tau)$ (but depending on $T$) such that
\begin{align}\label{2.dei2}
  H(u^j) &+ c\sum_{k=1}^j\tau\int_\Omega\bigg(\sum_{i=1}^2
  u_0^k|\pa_x(u_i^k)^{1/2}|^2 + |\pa_x u_0^k|^2 
  + |\pa_x (u_0^k)^{1/2}|^2\bigg)dx \\
  & + \eps C\sum_{k=1}^j\tau
  \sum_{i=1}^2 \|w_i^k\|_{H^1(\Omega)}^2 \le H(u^0) + C. \nonumber
\end{align}
\end{lemma}

\begin{proof}
We infer from \eqref{2.dei} and Lemma \ref{lem.h2A} that
\begin{align*}
  H(u^k)-H(u^{k-1}) &+ c\tau\int_\Omega\bigg(\sum_{i=1}^2
  u_0^k|\pa_x(u_i^k)^{1/2}|^2 + |\pa_x u_0^k|^2 
  + |\pa_x (u_0^k)^{1/2}|^2\bigg)dx \\
  &+ \eps\tau\sum_{i=1}^2 \|w_i^k\|_{H^1(\Omega)}^2 \le C\tau,
  \nonumber
\end{align*}
where $c>0$ depends only on $D_1,D_2$ and $C>0$ is independent of
$\eps,\tau$ and $k$. We sum this inequality over $k=1,\ldots,j$
and observe that $\tau j\le T$ to conclude the proof.
\end{proof}

\subsection{Uniform estimates}

We introduce the piecewise constant in time functions 
$u_i^{(\tau)}(x,t)$ $=u_i^k$ and $w_i^{(\tau)}=w_i^k$ for $x\in\Omega$, $t\in((k-1)\tau,k\tau]$, $i=1,2,$. We set $u^{(\tau)}(\cdot,0)=u^0$
and $w^{(\tau)}(\cdot,0)=h'(u^0)$ at time $t=0$. Furthermore, we
introduce the shift operator $(\sigma_\tau u^{(\tau)})(\cdot,t)
=u^{k-1}$ for $t\in((k-1)\tau,k\tau]$. Summing \eqref{2.approx} over $k=1,\ldots,N$ and using the definitions of $B(w^{(\tau)})$ and $w^{(\tau)}$, we infer that the pair $(u^{(\tau)},w^{(\tau)})$ solves
\begin{align}
  \frac{1}{\tau}\int_0^T&\int_\Omega 
  (u_i^{(\tau)}-\sigma_\tau u_i^{(\tau)})\phi_i dxdt 
  + \eps\int_0^T\int_\Omega(\pa_x w_i^{(\tau)}\pa_x\phi_i
  + w^{(\tau)}_i\phi_i)dxdt \label{2.tau} \\
  &+ D_i\int_0^T\int_\Omega(u_0^{(\tau)}\pa_x u_i^{(\tau)}
  - u_i^{(\tau)}\pa_x u_0^{(\tau)} - u_0^{(\tau)}u_i^{(\tau)}\pa_x V_i)
  \pa_x\phi_i dxdt \nonumber \\
  &+ \int_0^T\big(J_i^1[u^{(\tau)}](t)\phi_i(1,t)
  - J_i^0[u^{(\tau)}](t)\phi_i(0,t)\big)dt
  = 0, \nonumber
\end{align}
where $\phi_i:(0,T)\to H^1(\Omega)$ is piecewise constant, $i=1,2$, and $J_i^j[u^{(\tau)}](t)$ is evaluated at the time points $\lceil t/\tau\rceil\tau$, which means, for instance, 
$$
  J_1^0[u^{(\tau)}](t) = \alpha_1\Lambda_s(k\tau)u_0^{(\tau)}(1,t)
  \quad\mbox{for }t\in((k-1)\tau,k\tau].
$$
The discrete entropy inequality gives the following 
uniform bounds.

\begin{lemma}[Gradient bounds]\label{lem.grad}
There exists $C>0$ independent of $(\eps,\tau)$ such that
\begin{align*}
  \sum_{i=1}^2\big\|(u_0^{(\tau)})^{1/2}u_i^{(\tau)}
  \big\|_{L^2(0,T;H^1(\Omega))}
  + \|(u_0^{(\tau)})^{1/2}\|_{L^2(0,T;H^1(\Omega))} &\le C, \\
  \sum_{i=1}^2\|u_0^{(\tau)}u_i^{(\tau)}\|_{L^2(0,T;H^1(\Omega))}
   + \|u_0^{(\tau)}\|_{L^2(0,T;H^1(\Omega))} &\le C.
\end{align*}
\end{lemma}

\begin{proof}
The first estimate follows from the bound $0\le u_i^{(\tau)}\le 1$ and 
\eqref{2.dei} since
$$
  \big|\pa_x\big((u_0^{(\tau)})^{1/2} u_i^{(\tau)}\big)\big|
  \le \big|(u_0^{(\tau)})^{1/2}\pa_x u_i^{(\tau)}\big|
  + \big|u_i^{(\tau)}\big|\big|\pa_x(u_0^{(\tau)})^{1/2}\big|.
$$
We deduce from the first estimate and 
\begin{align*}
  |\pa_x(u_0^{(\tau)}u_i^{(\tau)})|
  &\le \big|(u_0^{(\tau)})^{1/2}
  \pa_x((u_0^{(\tau)})^{1/2}u_i^{(\tau)})\big|
  + \big|(u_0^{(\tau)})^{1/2}u_i^{(\tau)}
  \pa_x(u_0^{(\tau)})^{1/2}\big| \\
  &\le \big|\pa_x((u_0^{(\tau)})^{1/2}u_i^{(\tau)}\big|
  +  \big|\pa_x(u_0^{(\tau)})^{1/2}\big|,
\end{align*}
the second estimate.
\end{proof}

\begin{lemma}[Discrete time bounds]\label{lem.time}
There exists $C>0$ independent of $(\eps,\tau)$ such that
$$
  \|u_i^{(\tau)} - \sigma_\tau u_i^{(\tau)}\|_{L^2(0,T;H^1(\Omega)')}
  \le C\tau, \quad i=1,2.
$$
\end{lemma}

\begin{proof}
Let $\phi_i:(0,T)\to H^1(\Omega)$ be piecewise constant. Then, by \eqref{2.tau} and the $L^\infty(\Omega_T)$ bound of $u_i^{(\tau)}$,
\begin{align}\label{2.time}
  \frac{1}{\tau}&\bigg|\int_0^T\int_\Omega(u_i^{(\tau)}-\sigma_\tau
  u_i^{(\tau)})\phi_i dxdt\bigg| \\
  &\le D_i\big(\|(u_0^{(\tau)})^{1/2}\pa_x u_i^{(\tau)}\|_{L^2(\Omega_T)}
  + \|\pa_x u_0^{(\tau)}\|_{L^2(\Omega_T)} 
  + \|\pa_x V_i\|_{L^2(\Omega_T)}\big)\|\pa_x\phi_i\|_{L^2(\Omega_T)} 
  \nonumber \\
  &\phantom{xx}+ \sum_{j=0}^1\|J_i^j[u^{(\tau)}]\|_{L^2(0,T)}
  \|\phi_i\|_{L^2(0,T;H^1(\Omega))}
  + \eps\|w_i^{(\tau)}\|_{L^2(0,T;H^1(\Omega))}
  \|\phi_i\|_{L^2(0,T;H^1(\Omega))} \nonumber \\
  &\le C\|\phi_i\|_{L^2(0,T;H^1(\Omega))}.\nonumber
\end{align}
The last step follows from the boundedness of $J_i^j[u^{(\tau)}]$,
since $0\le u_i^{(\tau)}(x,t)\le 1$ for $x\in[0,1]$ and 
$0\le\Lambda_{n/s}(t)\le 1$. Inequality \eqref{2.time} holds for 
all piecewise constant functions $\phi_i:(0,T)\to H^1(\Omega)$. By a density argument, we obtain 
$$
  \tau^{-1}\|u_i^{(\tau)}-\sigma_\tau u_i^{(\tau)}\|_{L^2(0,T;H^1(\Omega)')}\le C,
$$
concluding the proof.
\end{proof}

\subsection{Limit $(\eps,\tau)\to 0$}

Lemmas \ref{lem.grad} and \ref{lem.time} allow us to apply the Aubin--Lions lemma in the version of \cite{DrJu12}, giving the existence of a subsequence, which is not relabeled, such that as $(\eps,\tau)\to 0$,
$$
  u_0^{(\tau)}\to u_0\quad\mbox{in }L^2(\Omega_T),
$$
and because of the uniform $L^\infty(\Omega_T)$ bound, this convergence holds in any $L^p(\Omega_T)$ for $p<\infty$. Moreover, we conclude the following weak convergences (up to subsequences):
\begin{align*}
  u_i^{(\tau)}\rightharpoonup u_i &\quad\mbox{weakly* in }
  L^\infty(\Omega_T), \\
  \tau^{-1}(u_i^{(\tau)}-\sigma_\tau u_i^{(\tau)}) \rightharpoonup
  \pa_t u_i &\quad\mbox{weakly in }L^2(0,T;H^1(\Omega)), \\
  \eps w_i^{(\tau)} \to 0 &\quad\mbox{strongly in }L^2(0,T;H^1(\Omega)).
\end{align*}
Since both $(u_i^{(\tau)})$ and $(\pa_x u_0^{(\tau)})$ are only weakly converging, we cannot obtain the convergence of the product. However,
the uniform bounds for $((u_0^{(\tau)})^{1/2}u_i^{(\tau)})$ and $((u_0^{(\tau)})^{1/2})$ in $L^2(0,T;H^1(\Omega))$ allow us to apply the ``degenerate'' version of the Aubin--Lions lemma \cite{BDPS10,Jue15} so that (for a subsequence)
$$
  (u_0^{(\tau)})^{1/2}u_i^{(\tau)}\to \sqrt{u_0}u_i 
  \quad\mbox{strongly in }L^p(\Omega_T),\ p<\infty
  \mbox{ as }(\eps,\tau)\to 0.
$$
This shows that
\begin{align*}
  u_0^{(\tau)}\pa_x u_i^{(\tau)} - u_i^{(\tau)}\pa_x u_0^{(\tau)}
  &= (u_0^{(\tau)})^{1/2}\pa_x\big((u_0^{(\tau)})^{1/2}u_i^{(\tau)}\big)
  - 3(u_0^{(\tau)})^{1/2}u_i^{(\tau)}\pa_x(u_0^{(\tau)})^{1/2} \\
  &\rightharpoonup \sqrt{u_0}\pa_x(\sqrt{u_0}u_i)
  - 3\sqrt{u_0}u_i\pa_x\sqrt{u_0}
\end{align*}
weakly in $L^1(\Omega_T)$, and since this sequence is bounded in $L^2(\Omega_T)$, the convergence holds true in that space. 

It follows from the linearity and continuity of the trace operator $H^1(\Omega)\to L^2(\pa\Omega)$ that this operator is weakly continuous and therefore,
$$
  u_0^{(\tau)}(x,\cdot)\to u_0(x,\cdot),\quad 
  (u_0^{(\tau)}u_i^{(\tau)})(x,\cdot)\rightharpoonup 
  (u_0u_i)(x,\cdot)\quad\mbox{weakly in }L^2(0,T),\ x=0,1.
$$ 
In fact, these sequences are even bounded in $L^\infty(0,T)$ because of the embedding $H^1(\Omega)\hookrightarrow C^0(\overline\Omega)\hookrightarrow L^\infty(\pa\Omega)$.
Let $\Lambda_{j}^{(\tau)}$ be the solution to \eqref{1.ode1} if $j=n$ or \eqref{1.ode2} if $j=s$ with $u$ replaced by $u^{(\tau)}$. Then  $\Lambda_{n}^{(\tau)}$ solves the integral equation
\begin{align*}
  \Lambda_n^{(\tau)} &= \Lambda_n(0) + \beta_1\int_0^t
  (1-\Lambda_n^{(\tau)}(r))
  \sigma_\tau(u_0^{(\tau)}u_1^{(\tau)})(1,r)dr 
  - \alpha_2\int_0^t\Lambda_n^{(\tau)}(r)\sigma_\tau u_0^{(\tau)}(1,r)dr.
\end{align*}
Since the integrand is uniformly bounded, this gives
$|\Lambda_n^{(\tau)}(t)-\Lambda_n^{(\tau)}(s)|\le C|t-s|$ for $s,t\in[0,T]$. Thus, $(\Lambda_n^{(\tau)})$ is uniformly bounded and uniformly equicontinuous. By the Arzel\`a--Ascoli theorem, there exists a subsequence (not relabeled) such that $\Lambda_n^{(\tau)}\to\Lambda_n$ uniformly in $[0,T]$. In a similar way, we prove that $\Lambda_s^{(\tau)}\to \Lambda_s$ uniformly in $[0,T]$. 
We need to identify the limits $\Lambda_n$ and $\Lambda_s$ as the solutions to \eqref{1.ode1} and \eqref{1.ode2}, respectively.

Set $G^{(\tau)}(t):=\Lambda_n^{(\tau)}(k\tau)$ for $t\in((k-1)\tau,k\tau]$. Then, for instance,
$$
  J_1^1[u^{(\tau)}](t) = \beta_1(1-G^{(\tau)}(t))u_0^{(\tau)}(1,t)u_1^{(\tau)}(1,t)
  \quad\mbox{for }t\in((k-1)\tau,k\tau].
$$
It holds for $s\in((m-1)\tau,m\tau]$ and $t\in((k-1)\tau,k\tau]$ that
$$
  |G^{(\tau)}(t)-G^{(\tau)}(s)|
  \le C|m\tau-k\tau|\le C(|t-s|+\tau).
$$
Therefore, since $G^{(\tau)}(\lceil t/\tau\rceil\tau)
= \Lambda_n^{(\tau)}(\lceil t/\tau\rceil\tau)$,
\begin{align*}
  |G^{(\tau)}(t)-\Lambda_n(t)|
  &\le |G^{(\tau)}(t)-G^{(\tau)}(\lceil t/\tau\rceil\tau)|
  + |\Lambda_n^{(\tau)}(\lceil t/\tau\rceil\tau)-\Lambda_n^{(\tau)}(t)|
  + |\Lambda_n^{(\tau)}(t)-\Lambda_n(t)| \\
  &\le C|t-\lceil t/\tau\rceil\tau| + C\tau 
  + \|\Lambda_n^{(\tau)}(t)-\Lambda_n(t)\|_{L^\infty(0,T)}\to 0
\end{align*}
as $(\eps,\tau)\to 0$, and this convergence is uniform in $[0,T]$. Hence, for instance,
$$
  J_1^1[u^{(\tau)}] \to \beta_1(1-\Lambda_n)u_0(1,\cdot)u_i(1,\cdot)
  =: J_1^1[u]\quad\mbox{strongly in }L^2(0,T).
$$
To establish that $\Lambda_n$ satisfies \eqref{1.ode1} it is sufficient to show that $\sigma_\tau(u_0^{(\tau)}u_i^{(\tau)})(x,\cdot)
\rightharpoonup (u_0u_1)(x,\cdot)$, $\sigma_\tau u_0^{(\tau)}(x,\cdot)\rightharpoonup u_0(x,\cdot)$ 
weakly in $L^2(0,T)$ for $x=0,1$. In fact, this result can be proved by straightforward arguments.
Then the convergence of $u_i^{(\tau)}(1,\cdot)$ in $L^2(0,T)$ implies that $\Lambda_n$ solves \eqref{1.ode1}. In a similar way, we prove that $\Lambda_s^{(\tau)}\to \Lambda_s$ uniformly in $[0,T]$, and $\Lambda_s$ solves \eqref{1.ode2}. 

The initial condition \eqref{1.ic}, understood in the sense of $H^1(\Omega)'$, follows from arguments similar as at the end of the proof of Theorem 2 in \cite{Jue15}. This finishes the proof.


\section{Numerical experiments and stationary states}\label{sec.num}

\subsection{Numerical scheme and parameters}

We discretize equations \eqref{1.eq1}--\eqref{1.eq2} by an implicit Euler finite-volume scheme. Let $n,m\in\N$ and set $\tau=T/n$, $h=1/m$. We divide $\Omega=(0,1)$ into $m$ cells $(x_{j},x_{j+1})$ for $j=0,\ldots,m-1$, where $x_j=jh$. (Note that the notation is different from Section \ref{sec.deriv}.) We approximate $h^{-1}\int_{x_j}^{x_{j+1}}u_i(x,k\tau)dx$ by $u_{i,j}^k$, which solves
for $k=1,\ldots,n$,
\begin{align*}
  u_{i,j}^k &= u_{i,j}^{k-1} 
  + \frac{\tau}{h}(J_{i,j+1/2}^k-J_{i,j-1/2}^k), 
  \quad i=1,2,\ j=1,\ldots,m-1, \\
  J_{i,j+1/2}^k &= -\frac{D_i}{h}\big(
  \bar{u}_{0,j+1/2}^k(u_{i,j+1}^k-u_{i,j}^k)
  + \bar{u}_{i,j+1/2}(u_{0,j+1}^k-u_{0,j}^k)\big) \\
  &\phantom{xx}- D_i\bar{u}_{0,j+1/2}^k\bar{u}_{i,j+1/2}^k
  \pa_x V_i(x_{j+1/2}),
\end{align*}
where $\bar{u}_{i,j+1/2}^k:=(u_{i,j+1}^k+u_{i,j}^k)/2$ for $i=0,1,2$.
At the boundary points $x=0$ and $x=1$, we replace $J_{i,1/2}^k$ and $J_{i,m-1/2}^k$ respectively, by the corresponding boundary condition, evaluated at $x_0=0$ or $x_m=1$ and at time $k\tau$. For instance,
$J_{1,0}^k = \alpha_1\Lambda_s(k\tau)u_{0,0}^k$. The differential equations \eqref{1.ode1}--\eqref{1.ode2} are discretized by the implicit Euler scheme, for instance,
$$
  \Lambda_s^k = \Lambda_s^{k-1} 
  - \tau\alpha_1\frac{\Lambda_s^k}{\Lambda_s^{\rm max}} u_{0,0}^k
  + \tau\beta_2\bigg(1-\frac{\Lambda_s^k}{\Lambda_s^{\rm max}}\bigg)
  u_{0,0}^k u_{2,0}^k.
$$

The nonlinear discrete system is solved by using a damped Newton method.
More precisely, let $F:\R^{3m+2}\to\R^{3m+2}$ be given by
\begin{align*}
  F_{j+m(i-1)}(y) &= u_{i,j}^{k-1} + \frac{\tau}{h}(J_{i,j+1/2}^k
  - J_{i,j-1/2}^k) - y_{j+m(i-1)}, \quad i=1,2, \\
  F_{j+2m}(y) &= y_{j+2m} - y_{j+m} - y_j, \\
  F_{j+2m+2}(y) &= \Lambda_s^{k-1} - \tau\alpha_1
  \frac{y_{3m+2}}{\Lambda_s^{\rm max}}(1-y_{2m+1}) \\
  &\phantom{xx}
  + \tau\beta_2\bigg(1-\frac{y_{3m+2}}{\Lambda_s^{\rm max}}\bigg)
  (1-y_{2m+1})y_{m+1} - y_{3m+2},
\end{align*}
where $y=(y_1,\ldots,y_{3m+2})\in\R^{3m+2}$
and $F_{3m+1}(y)$ is defined similarly from the implicit Euler scheme for $\Lambda_n^k$. The damped Newton method reads as
$$
  y^{(r+1)} = y^{(r)} + \frac{1}{(r+1)^{3/4}}
  \frac{\widehat{y}^{(r+1)}}{\|\widehat{y}^{(r+1)}\|_\infty},
  \quad r\in\N,
$$
where $\widehat{y}^{(r+1)}$ solves $F'(y^{(r)})(\widehat{y}^{(r+1)}-
y^{(r)}) = -F(y^{(r)})$. The exponent $3/4$ was determined from numerical experiments. We stopped the Newton iterations when 
$\|F(y^{(r)})\|_\infty<\eps$ with $\eps=10^{-3}$ is reached.
The numerical scheme is implemented in Python version 3.7.1.
We collect the values of the parameters, inspired from 
\cite{HDPP21}, in Table \ref{table}. If not otherwise stated, we set $h=0.0025$ and $\tau=10^{-4}$. 

\begin{table}[ht]
  \begin{tabular}{|l|l||l|l||l|l|}
  \hline
  $\alpha_1$ & 0.2666 & $\Lambda_n^{\rm{max}}$ & 0.0029 &
  $D_1$ & 0.0004 \\
  $\alpha_2$ & 0.2666 & $\Lambda_n^0$ & 0.0015 &
  $D_2$ & 0.004 \\
  $\beta_1$ & 3 & $\Lambda_s^{\rm{max}}$ & 0.175 &
  $V_1(x)$ & $1.75x$ \\
  $\beta_2$ & 3 & $\Lambda_s^0$ & 0.12 &
  $V_2(x)$ & $-1.5x$ \\ \hline
  \end{tabular}
  \caption{Numerical parameters.}
  \label{table}
\end{table}

\subsection{Numerical experiment 1}

We choose the initial data $u_1^0=u_2^0=0.1$. Figure \ref{fig1} presents the vesicle concentrations at times $t=0,1,10$ and the evolution of the number $\Lambda_n(t)$ of vesicles in the growth cone. The anterograde vesicles (species 1) are leaving the soma, leading to an increase of the concentration near $x=0$, while it is decreasing near the tip of the neurite at $x=1$ because of the small value of $\Lambda_s$. The retrograde vesicles (species 2) are leaving the growth cone at $x=1$, leading to an increase of the concentration, while it is decreasing near the soma. The number $\Lambda_s$ is decreasing over time, which can be explained by the difference of magnitude of the parameters $\alpha_1$ and $\beta_2$ governing the outflow rate. 

The behavior of the vesicles at $t=10$ in our model and the model of \cite{HDPP21} is similar; see the middle row of Figure \ref{fig1}. The difference is largest near the growth cone at $x=1$ (see the bottom left panel), which comes from the different boundary conditions at this point. Since the boundary value $J_1^1[u]$ contains the factor $u_0<1$ in our model, the number $\Lambda_n$ is decreasing at a faster rate compared to the model of \cite{HDPP21} (see the bottom right panel). 

\begin{figure}[ht]
\begin{center}
\hspace*{45mm} $t=0$ \hfill $t=1$ \hspace*{35mm}
\end{center}
\includegraphics[width=70mm, height=50mm]{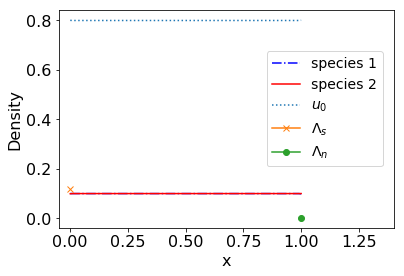} 
\includegraphics[width=70mm, height=50mm]{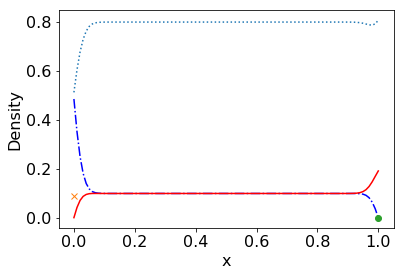} 
\begin{center}
\hspace*{45mm} $t=10$ \hfill $t=10$, model of \cite{HDPP21}\hspace*{20mm}
\end{center}
\includegraphics[width=70mm, height=50mm]{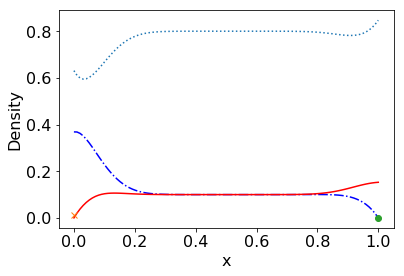} 
\includegraphics[width=70mm, height=50mm]{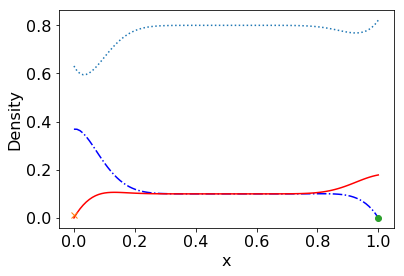} 
\includegraphics[width=70mm, height=50mm]{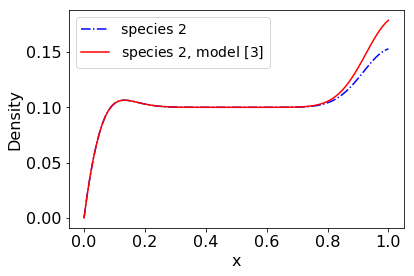}
\includegraphics[width=70mm, height=50mm]{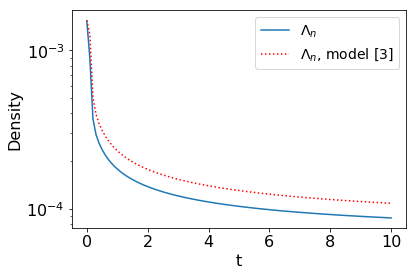}
\caption{Experiment 1: Concentrations of anterograde vesicles (species 1) and retrograde vesicles (species 2). Top row: $t=0,1$. Middle row: $t=10$. Bottom left: $t=10$, only species 2. Bottom right: Evolution of $\Lambda_n(t)$.}
\label{fig1}
\end{figure}

\subsection{Numerical experiment 2}

In this example, we choose piecewise constant initial data:
$$
  u_1^0(x) = \left\{\begin{array}{ll}
  0.9 & \quad\mbox{for }0.1<x<0.4, \\
  0 & \quad\mbox{else},
  \end{array}\right. \quad
  u_2^0(x) = \left\{\begin{array}{ll}
    0.9 & \quad\mbox{for }0.6<x<0.9, \\
    0 & \quad\mbox{else},
    \end{array}\right.
$$
The numerical results at times $t=0,1,10,100$ are shown in Figure \ref{fig2}. We observe a smoothing effect (due to diffusion) and a drift of the vesicles profiles towards the middle. 
The drift of the anterograde vesicles is stronger compared to the retrograde vesicles because of $|\pa_x V_1|>|\pa_x V_2|$. Since the boundary values of the vesicles are very small, the results of our model are almost identical to those from the model of \cite{HDPP21}; see Figure \ref{fig2} bottom for $\Lambda_n$ and $\Lambda_s$ up to $t=10$.

\begin{figure}[ht]
\begin{center}
\hspace*{45mm} $t=0$ \hfill $t=1$ \hspace*{35mm}
\end{center}
\includegraphics[width=70mm, height=50mm]{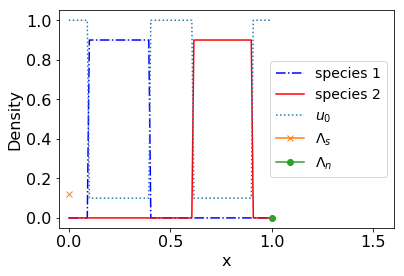}
\includegraphics[width=70mm, height=50mm]{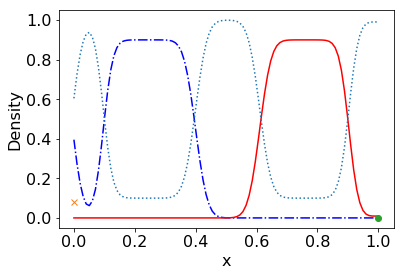} 
\begin{center}
\hspace*{45mm} $t=10$ \hfill $t=100$ \hspace*{35mm}
\end{center}
\includegraphics[width=70mm, height=50mm]{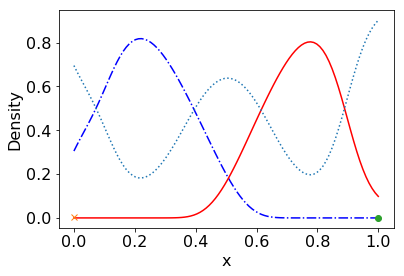}
\includegraphics[width=70mm, height=50mm]{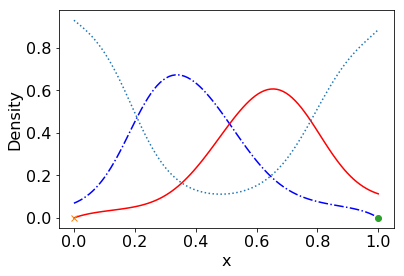} 
\includegraphics[width=70mm, height=50mm]{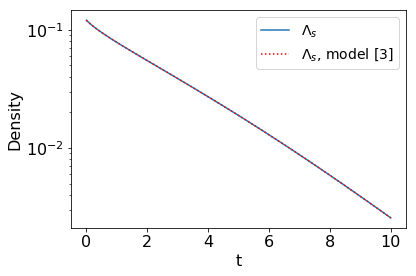} 
\includegraphics[width=70mm, height=50mm]{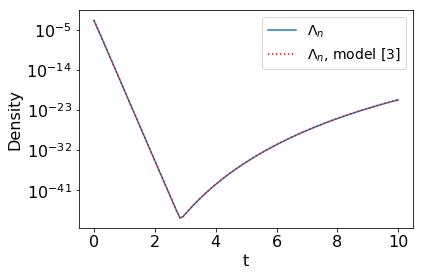} 
\caption{Experiment 2: Concentrations of anterograde vesicles (species 1) and retrograde vesicles (species 2). Top and middle rows: $t=0,1,10,100$. Bottom row: evolution of $\Lambda_s$ (left) and $\Lambda_n$ (right).}
\label{fig2}
\end{figure}

\subsection{Convergence rates}

We test our numerical scheme by computing the spatial and temporal convergence rates. We choose the initial data $u_1^0=u_2^0=0.1$ and the parameters from Table \ref{table}. Furthermore, we set $T=1$. We define the mean error as the discrete $L^2$ norm 
$\|u-u^{\rm ref}\|_2/\sqrt{2(m+1)}$, where
$u=(u_1,u_2,\Lambda_n,\Lambda_s)$ and $u^{\rm ref}=
(u_1^{\rm ref},u_2^{\rm ref},\Lambda_n^{\rm ref},\Lambda_s^{\rm ref})$ is the reference solution. 

Figure \ref{fig.conv} (left) shows the discrete $L^2$ error for time step sizes $\tau=10^{-2}\cdot 2^{-k}$ for $k=1,\ldots,7$ with fixed $h=10^{-3}$. The reference solution is computed with $h=10^{-3}$ and $\tau=10^{-5}$. The convergence is of first order for rather large values of $\tau$, while it is between first and second order when the time step size is closer to the step size of the reference solution.
The spatial convergence is illustrated in Figure \ref{fig.conv} (right) for grid sizes $h=10^{-2}\cdot 2^{-k}$ for $k=1,\ldots,7$ with fixed $\tau=10^{-3}$. The reference solution is calculated by using the parameters $h=10^{-5}$ and $\tau=10^{-3}$. The convergence is of first order (if $\tau$ is not too large), which is expected for the two-point approximation finite-volume scheme. 

\begin{figure}[ht]
\includegraphics[width=75mm, height=55mm]{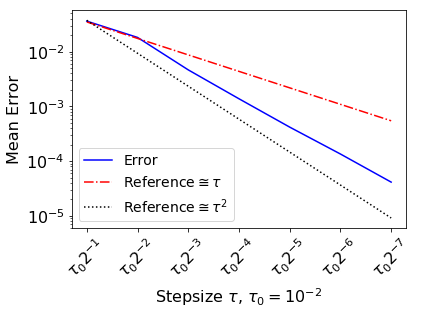} 
\includegraphics[width=75mm, height=55mm]{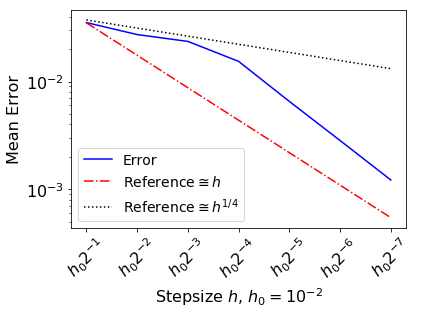} 
\caption{Left: Discrete $L^2$ error versus time step size $\tau$ for fixed $h=10^{-3}$. Right: Discrete $L^2$ error versus space step size $h$ for fixed $\tau=10^{-3}$.}
\label{fig.conv}
\end{figure}


\subsection{Stationary states}\label{sec.steady}

In this section, we derive some properties of stationary solutions, i.e., solutions $(u_1,u_2,$ $\Lambda_n,\Lambda_s)$ to \eqref{1.eq1}--\eqref{1.ode2}, where $\pa_t u_1=\pa_t u_2=0$ and $\pa_t\Lambda_n=\pa_t\Lambda_s=0$. The former condition implies that the fluxes $J_1$ and $J_2$ are constant, and we deduce from the latter condition that the total flux vanishes, $J_1+J_2=0$. Consequently, $J:=J_1=-J_2$. Moreover, if $u_0(1)>0$ and $u_0(0)>0$, the stationary solution to \eqref{1.ode1}--\eqref{1.ode2} is given by
\begin{equation}\label{4.lamb}
  \Lambda_n = \frac{\beta_1 u_1(1)}{\beta_1 u_1(1)+\alpha_2}, \quad \Lambda_s = \frac{\beta_2 u_2(0)}{\beta_2 u_2(0)+\alpha_1},
\end{equation}
We assume that a stationary solution exists and that $u_1,u_2\in W^{1,\infty}(\Omega)$. Then
\begin{equation}\label{4.J}
  J = -D_1\big(u_0\pa_x u_1-u_1\pa_x u_0-u_0u_1\pa_x V_1\big) = D_2\big(u_0\pa_x u_2-u_2\pa_x u_0-u_0u_2\pa_x V_2\big).
\end{equation}

The following situation is approximately satisfied in numerical experiment 1 for large times.

\begin{lemma}\label{lem1}
Let $u_0(1)>0$ and $u_0(0)>0$. Then $\Lambda_n=0$ if and only $\Lambda_s=0$, and $u_1(1)=0$ if and only of $u_2(0)=0$. 
In this situation, the flux vanishes, $J=0$. 
\end{lemma}
  
\begin{proof}
Let $\Lambda_n=0$. Then, by \eqref{4.lamb}, $u_1(1)=0$. We insert expressions \eqref{4.lamb} into the boundary conditions \eqref{1.bc10}--\eqref{1.bc11}:
\begin{equation}\label{4.aux}
  J = J_1(0) = \frac{\alpha_1\beta_2 u_2(0)}{\beta_2 u_2(0) + \alpha_1}u_0(0) = J_1(1) = \frac{\alpha_2\beta_1 u_1(1)}{\beta_1 u_1(1)+\alpha_2}u_0(1) = 0.
\end{equation}
This shows that $u_2(0)=0$ and consequently, again by \eqref{4.lamb}, $\Lambda_s=0$. Moreover, we infer from \eqref{4.aux} that $J=0$. 
\end{proof}

If the parameters are the same for both species, the solution is symmetric around $x=1/2$, as proved in the following lemma.

\begin{lemma}\label{lem2}
Let $\alpha_1=\alpha_2$, $\beta_1=\beta_2$, $\Lambda_n^{\rm max}=\Lambda_s^{\rm max}$, $D_1=D_2$, and $V_2(x)=V_1(1-x)+\text{const.}$ for $x\in\Omega$. Then $(u_1,u_2,\Lambda_n,\Lambda_s)$ with
$u_2(x)=u_1(1-x)$ for $x\in\Omega$ and  $\Lambda_n=\Lambda_s$ is a stationary solution to \eqref{1.eq1}--\eqref{1.bc21}. 
\end{lemma}

\begin{proof}
Let $u_1$ be a solution to \eqref{4.J} with $u_0:=1-u_{1}(x)-u_1(1-x)$
and $u_2(x):=u_1(1-x)$ for $x\in\Omega$. Taking into account that $\pa_x u_2(x) = -\pa_x u_1(1-x)$ and $\pa_x V_2(x) = -\pa_x V_1(1-x)$, we deduce from $u_0(x)=u_0(1-x)$ that
\begin{align*}
  -J/D_1 &= u_0(x)\pa_x u_1(x) - u_1(x)\pa_x u_0(x) 
  - u_0(x)u_1(x)\pa_x V_1(x) \\
  &= -u_0(1-x)\pa_x u_2(1-x) + u_2(1-x)\pa_x u_0(1-x) \\
  &\phantom{xx}+ u_0(1-x)u_2(1-x) \pa_x V_2(1-x).
\end{align*}
Thus, $(u_1,u_2)$ solves \eqref{4.J}. We infer from $u_1(1)=u_2(0)$ and \eqref{4.lamb} that $\Lambda_n=\Lambda_s$. Furthermore, since $u_0(0)=u_0(1)$, the boundary conditions \eqref{1.bc10}--\eqref{1.bc21} are satisfied.
\end{proof}

This situation is illustrated in Figure \ref{fig.steady}. 
We have chosen $\Lambda_{n}^{\max}=\Lambda_{s}^{\max}=0.175$, $\Lambda_{n}^0=\Lambda_{s}^0=0.12$, with potentials $V_{1}(x)=1.5x$, $V_{2}(x)=-1.5x$, and initial data $u_{1}^0=u_{2}^0=0.1$. The left panel shows the concentrations at $T=1000$ using the parameters $\alpha_i$, $\beta_i$, and $D_i$ as in Experiment 1. The solution is approximately stationary (the modulus of the flux is less than 0.01). 
Since $u_2(0)=0$, Lemma \ref{lem1} shows that the stationary flux vanishes. In the right panel, we present a case where the stationary flux does not vanish. Here, the solution is computed up to $T=100$, the parameters are $\alpha_i=\beta_i=D_i=1$ for $i=1,2$, and the flux equals $J=0.118$. 

\begin{figure}[ht]
\includegraphics[width=75mm, height=55mm]{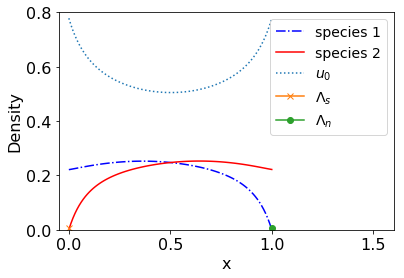} 
\includegraphics[width=75mm, height=55mm]{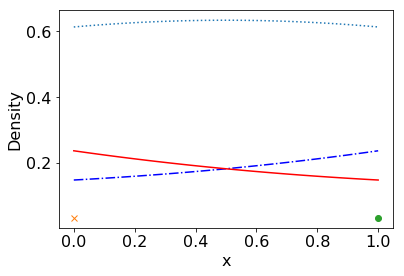} 
\caption{Concentrations of anterograde and retrograde vesicles.
Left: $J=0$. Right: $J\neq 0$.}
\label{fig.steady}
\end{figure}




\end{document}